%% file: magicgraph.tex
\documentclass{article}
\usepackage{spconf,amsmath,graphicx}
\usepackage{amsmath}
\usepackage{amsthm}
\usepackage{amssymb}
\usepackage{hyperref,url}
\usepackage{booktabs}
\usepackage[mathscr]{eucal}
\usepackage{graphicx}
\usepackage{mathtools} \mathtoolsset{showonlyrefs}
\usepackage{tabularx}
\usepackage{microtype}
\usepackage{subfigure}
\usepackage{color}
\newcommand\eps{\epsilon}
\input{math_commands}

\newtheorem{theorem}{Theorem}

\newtheorem{definition}{Definition}

\title{Sparse graph based sketching  for fast numerical linear algebra}
\name{Dong Hu$^\star$ \:\: Shashanka Ubaru$^\dagger$ \:\: Alex Gittens$^\star$ \:\: Kenneth L.~Clarkson$^\#$ \:\: Lior Horesh$^\dagger$ \:\: Vassilis Kalantzis$^\dagger$}
\address{$^\star$Rensselaer Polytechnic Institute, Troy, NY\\
$^\dagger$ IBM Research, Yorktown Heights, NY. $^\#$IBM Research, San Jose, CA\thanks{Copyright 2021 IEEE. Published in ICASSP 2021 - 2021 IEEE International Conference on Acoustics, Speech and Signal Processing (ICASSP), scheduled for 6-11 June 2021 in Toronto, Ontario, Canada.
Personal use of this material is permitted. However, permission to reprint/republish this material for advertising or promotional purposes or for creating new collective works for resale or redistribution to servers or lists, or to reuse any copyrighted component of this work in other works, must be obtained from the IEEE. Contact: Manager, Copyrights and Permissions / IEEE Service Center / 445 Hoes Lane / P.O. Box 1331 / Piscataway, NJ 08855-1331, USA. Telephone: + Intl. 908-562-3966.}}
\begin{document}
%\ninept
%
\maketitle
\begin{abstract}
In recent years, a variety of randomized constructions of \emph{sketching matrices} have been devised, that have been used in fast algorithms for numerical linear algebra problems, such as least squares regression, low-rank approximation, and the approximation of leverage scores. A key property of sketching matrices is that of \emph{subspace embedding}.
In this paper, we  study sketching matrices that are obtained from bipartite graphs that are sparse, i.e., have left degree~$s$ that is small. In particular, we explore two popular classes of sparse graphs, namely, \emph{expander graphs} and \emph{magical graphs}. For a given subspace $\mathcal{U} \subseteq \R^n$ of dimension $k$, we show that the magical graph with left degree $s=2$ yields a $(1\pm \eps)$ $\normltwo$-subspace embedding for  $\mathcal{U}$, if the number of right vertices (the sketch size) $m = \mathcal{O}({k^2}/{\eps^2})$. The expander graph with $s = \mathcal{O}({\log k}/{\eps})$ yields a subspace embedding for $m = \mathcal{O}({k \log k}/{\eps^2})$.
We also discuss the construction of sparse sketching matrices with reduced randomness using expanders based on error-correcting codes.
Empirical results on various synthetic and real datasets show that these sparse graph sketching matrices work very well in practice.

\end{abstract}
\section{Introduction}
\label{sec:intro}
Many modern applications involving large-dimensional data resort to data reduction techniques such as matrix approximation~\cite{halko2011finding, ubaru2019sampling} or compression~\cite{sayood2017introduction} 
%\LH{need to provide some more details here, projection methods ? compression based methods ? etc / or provide some citations} 
in order to lower the amount of data processed, and to speed up computations.
In recent years, randomized sketching techniques have attracted much attention in the literature for computing such matrix approximations, due to their high efficiency and well-established theoretical guarantees~\cite{KV17, halko2011finding,woodruff2014sketching}. 
Sketching methods have been used to speed up numerical linear algebra problems such as least squares regression, low-rank approximation, matrix multiplication, and  approximating leverage scores~\cite{sarlos2006improved,clarkson2017low,halko2011finding,woodruff2014sketching,ubaru2017low}. These primitives as well as improved matrix algorithms have been developed based on sketching for various tasks in machine learning~\cite{chatalic2018large,zhang2019incremental}, signal processing~\cite{choi2019large,rahmani2020scalable}, scientific computing~\cite{ubaru2019sampling}, and optimization~\cite{pilanci2017newton,ozaslan2019iterative}.

Given a large matrix $\mA\in \R^{n\times d}$, sketching techniques consider a random ``sketching" matrix $\mS\in \R^{m\times n}$ with $m\ll n$, where the matrix product $\mS\cdot \mA$ is called the sketch of $\mA$. Then, the sketch (smaller matrix) substitutes the original matrix in  applications.
For example, in over-constrained least-squares regression, where design matrix $\mA\in \R^{n\times d}$ has $n\gg d$, and with response vector $\vb\in\R^n$, we wish to solve $\min_{\vx\in\R^d}\|\mA\vx - \vb\|_2$ inexpensively, and obtain an approximate solution $\tilde{\vx}$ such that $\|\mA\tilde{\vx} - \vb\|_2\leq (1+\eps) \|\mA\vx^* - \vb\|_2$, where $\vx^* = (\mA^{\top}\mA)^{-1}\mA^{\top}\vb$ is the exact solution and $\eps$ is an error tolerance. Using sketching, we can obtain a fast approximate solution, by solving
$\tilde{\vx} = \argmin_{\vx\in\R^d}\|\mS\mA\vx - \mS\vb\|_2$.
Similarly, in low-rank approximation, given a large matrix $\mA\in \R^{n\times d}$ and a rank $k$, we wish to compute a rank-$k$ matrix $\tilde{\mA_k}$ for which $\|\mA-\tilde{\mA_k}\| \leq (1+\eps) \|\mA - {\mA_k}\|$, where $\mA_k$ is the best rank-$k$ approximation with respect to the matrix norm. Again, the matrix $\tilde{\mA_k}$ can be computed inexpensively using the sketch $\mS \mA$, see section~\ref{sec:results}. In both  cases, the $(1+\eps)$-approximation error bound is obtained if the sketching matrix $\mS$ is a
\emph{$(1\pm \eps)$ $\normltwo$-subspace embedding}~\cite{sarlos2006improved,woodruff2014sketching} (defined in section~\ref{sec:theory}).

Different variants of sketching matrices have been proposed in the literature~\cite{avron2013sketching,pagh2013compressed,woodruff2014sketching}.
Random matrices with i.i.d.  Gaussian entries~\cite{halko2011finding} and random sign matrices~\cite{clarkson2009numerical} were proposed, which for a rank-$k$ approximation require a sketch size of only $m = \mathcal{O}\big(\frac{k}{\epsilon}\big)$. However, these matrices are dense, require many i.i.d random numbers to be generated, and the cost of computing the  sketch $\mS \mA$ will be $\mathcal{O}(ndm)$. 
 Structured random matrices, like subsampled random Fourier transform  (SRFT), Hadamard  transform  (SRHT) and error-correcting code (ECC) matrices have been used to overcome the randomness and sketch cost issues~\cite{woolfe2008fast,boutsidis2013improved,ubaru2017low}.
 For dense $\mA$, the sketch cost will be  $\mathcal{O}(nd\log(m))$.
 However, the sketch size will need to be $m = \mathcal{O}\big(\frac{k\log k}{\epsilon}\big)$.

 Sparse sketching matrices have also been explored~\cite{clarkson2017low,nelson2013osnap} for sparse input matrices. \emph{Countsketch} matrices~\cite{clarkson2017low}, which has just one nonzero per column of $\mS$, have
 been shown to perform well for sketching. For sparse $\mA$, the sketching cost will be $\mathcal{O}(\nnz(\mA))$, that is, linear in the number of nonzeros of $\mA$.
 But, for a good rank-$k$ approximation, the sketch size will have to be
 $m = \mathcal{O}\big(\frac{k^2}{\epsilon^2}\big)$~\cite{woodruff2014sketching}.
 Applications involving low-rank approximation of large sparse matrices include Latent Semantic Indexing (LSI)~\cite{LSIsaad}, matrix completion problems~\cite{hastie2015matrix}, subspace tracking in signal processing~\cite{yang1995projection} and others~\cite{ubaru2019sampling}.

\begin{figure}[t]

\centering
\includegraphics[width=0.4\textwidth]{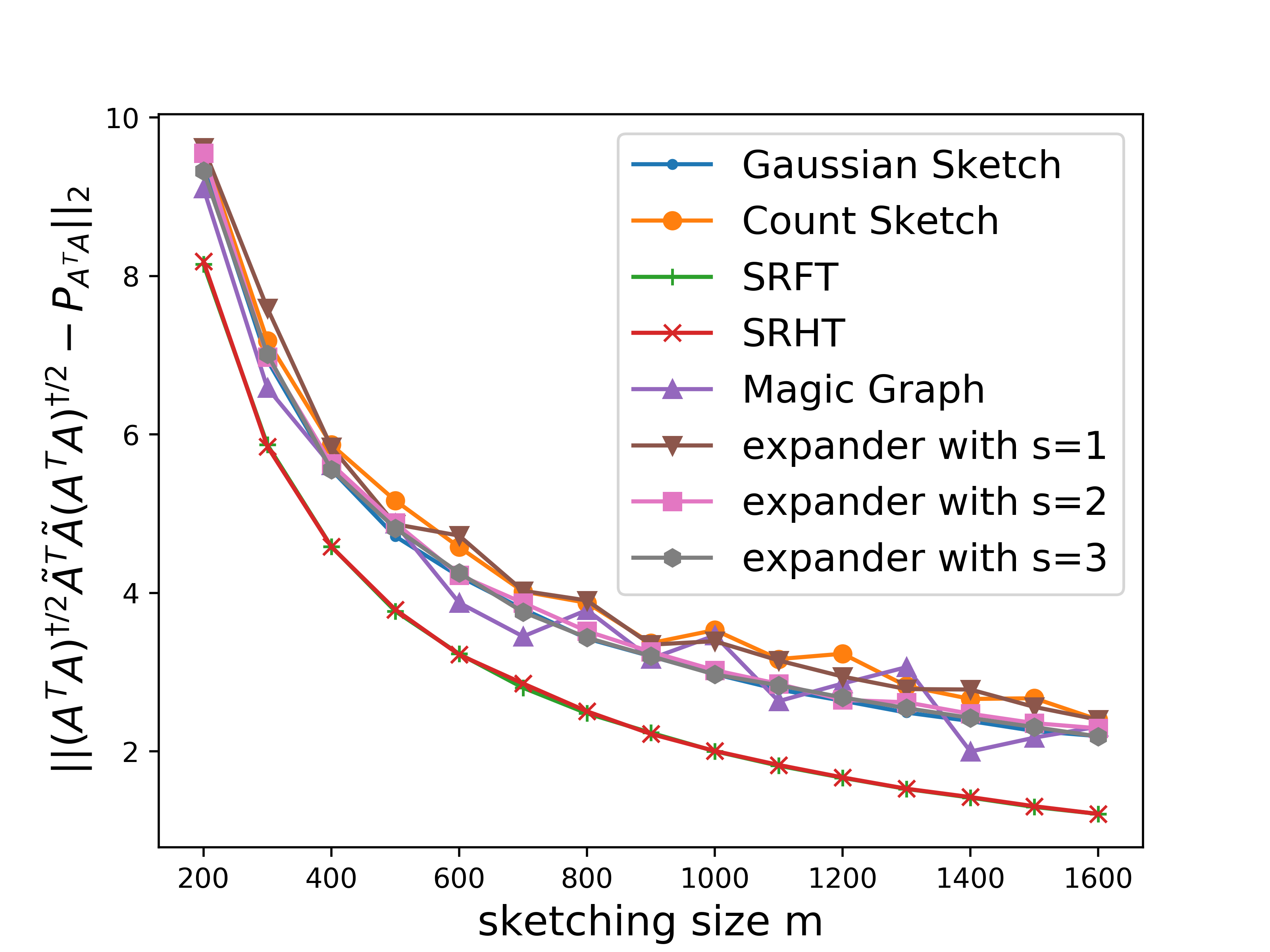}
\caption{Various sketching methods applied on a Gaussian matrix $\mA \in \mathbf R^{4096\times 1000}$ with entries $\mA_{i,j}\sim\mathcal{N}(0,1)$. The distortion (see Sec.~\ref{sec:results}) is plotted as a function of the sketch size $m$. Distortion decreases as $m$ is increased, and all the sketching methods perform similarly on this synthetic data. }\label{fig:1}
\end{figure}

In this study, we explore sparse sketching matrices that are obtained from  bipartite graphs. Our study is motivated in part by the numerical results we obtained while developing a \href{https://github.com/jurohd/Matrix-Sketching-Toolkits}{Matrix Sketching} toolkit. In particular, we observed that a sparse matrix obtained from so-called \emph{magical graphs}, used in~\cite{cheung2013fast} for fast rank estimation, performed very well as a sketching matrix in practice. Figure~\ref{fig:1} presents one such result, where we plot the distortion~\cite{magdon2019fast} (see, section~\ref{sec:results} for definition) between the input matrix $\mA$ and its sketch $\mS\mA$ when we applied different sketching matrices.
We observe that the matrices from sparse graphs such as magical graph and expander graphs perform as well as known sketching matrices. In this paper, we show that such matrices based on sparse graphs are indeed $(1\pm \eps)$ $\normltwo$-subspace embeddings.
Specifically, we present theoretical results which show that the magical graph matrix ($s=2$) of~\cite{cheung2013fast} will be a subspace embedding for
$m = \mathcal{O}\big(\frac{k^2}{\epsilon^2}\big)$. Matrices obtained from expander graphs with degree $s = \mathcal{O}\big(\frac{\log k}{\epsilon}\big)$ will be subspace embeddings for a sketch size $m =\mathcal{O}\big(\frac{k\log k}{\epsilon^2}\big)$. We also discuss how we can construct sparse sketching matrices with reduced randomness using expander graphs based on error correcting codes~\cite{guruswami2009unbalanced}. These can be viewed as alternate constructions for the sparse sketch matrices with limited randomness discussed in~\cite{nelson2013osnap}. Finally, we present numerical results that illustrate the performance of the various sketching matrices for low-rank approximation.

\section{Sparse Graphs}
First, we describe a class of sparse graphs known as \emph{expander graphs}, and variants called magical graphs.
\begin{definition}[Expander Graph]
A bipartite graph $\gG= (L, R;E)$  with $|L| = n,|R|=m$, and regular left degree $s$,
is a $(k, 1-\eps)$-expander graph if for any $C\subseteq L$ with $|C|\leq k$, the set of neighbors $\Gamma(C)$ of $C$ has size $|\Gamma(C)| > (1-\eps)s |C|$.
\end{definition}
Figure~\ref{fig:exp} gives an illustration of an expander graph. These expander graphs have been used in a number of applications including coding theory~\cite{sipser1996expander}, compressed sensing~\cite{jafarpour2009efficient},  group testing~\cite{cheraghchi2010derandomization}, multi-label classification~\cite{ubaru2017multilabel}, and others~\cite{hoory2006expander}. 
A number of random and deterministic constructions of expander graphs exist.
A uniformly random bipartite graph is an expander graph with high probability when $s = O\left(\frac{\log n/k}{\epsilon}\right)$ and $m =O\left(\frac{k\log (n/k)}{\epsilon^2}\right)$~\cite{hoory2006expander,cheraghchi2010derandomization,ubaru2017multilabel}. That is, for each left vertex, we choose $s$ right vertices uniformly and  independently at random.

\begin{figure}[tb]

\centering
\includegraphics[width=0.4\textwidth,trim={5cm 2cm 5cm 3cm}]{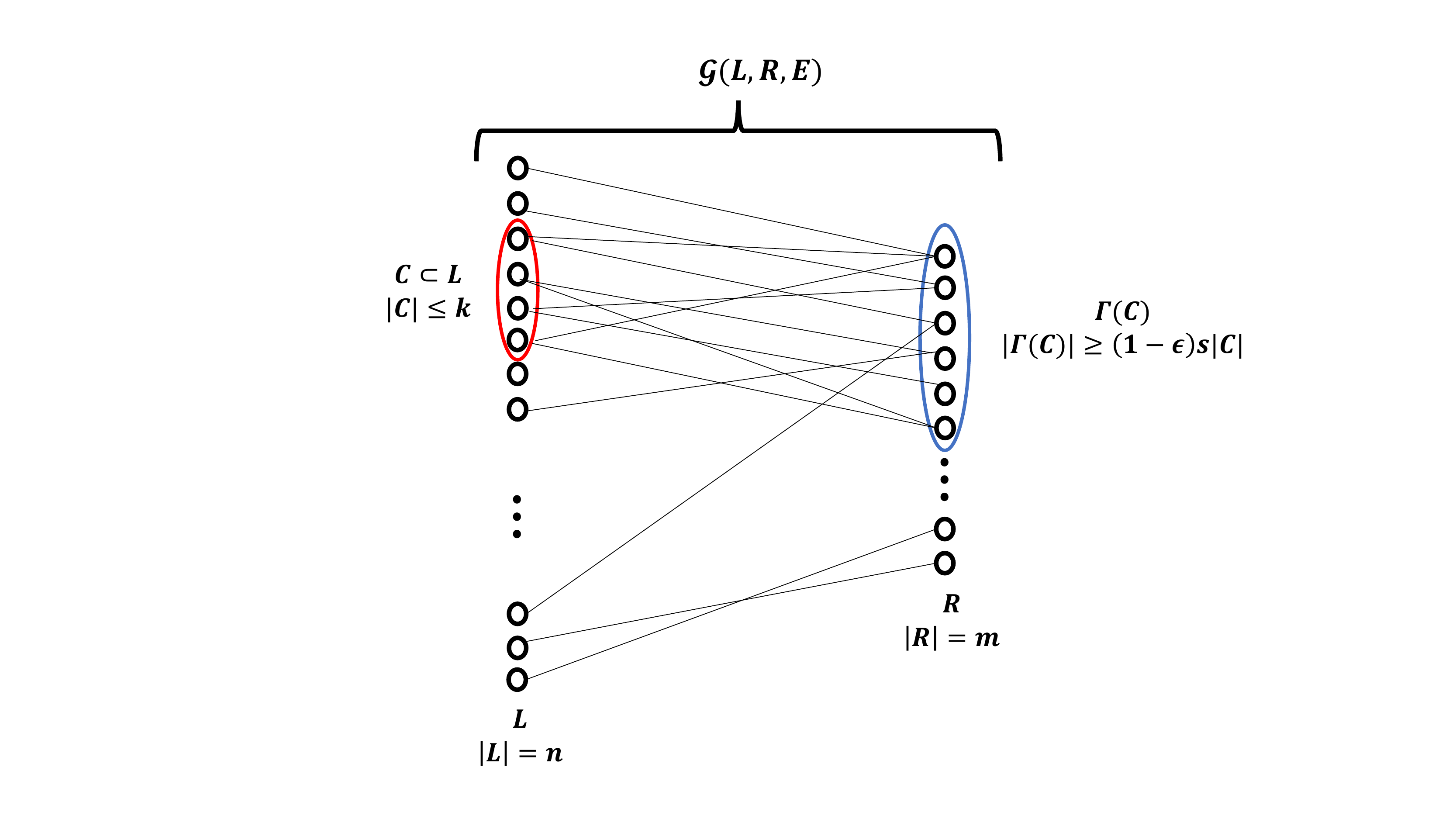}
\caption{$(k,1-\eps)$ -Expander graph}\label{fig:exp}
\end{figure}
Explicit constructions of expander graphs are also known~\cite{capalbo2002randomness,reingold2000entropy}, and there are many constructions of expander graphs based on error-correcting codes~\cite{guruswami2009unbalanced,alon1994random,cheraghchi2010derandomization}. Such constructions can yield sparse sketching matrices with reduced randomness.
A few constructions of subspace-embedding sketching matrices based on error-correcting codes are known~\cite{ubaru2017low}.

Magical graphs are  variants of expander graphs that have been used to construct super-concentrators~\cite{hoory2006expander}. Cheung et.~al~\cite{cheung2013fast} showed that such graphs (defined below) can be used to develop fast algorithms for matrix rank estimation. 

\begin{definition}[Magical Graph]
A random bipartite graph $\gG= (L, R;E)$ with $|L| = n,|R|=m$ is $(k, \delta)$-magical if for every $C\subseteq L$ with $|C|=k$, the probability that there is a matching in $\gG$ in which every vertex in $C$ is matched is at least $1 - \delta$.
\end{definition}
Article~\cite{cheung2013fast} also presented a random bipartite graph construction with $s=2$ that is $(k, O(1/k))$-magical, when $m\geq 11k$. We can show that this construction is also  a $(k, 1-\eps)$-expander, if $m = \mathcal{O}(k^{1/\eps})$.

Here we consider the incidence/adjacency matrix $\mS$ corresponding to such a bipartite graph as our sketching matrix, with the nonzero entries (edges) randomly set to $\pm 1/\sqrt{s}$. Note that, with the sketch matrix based on an expander graph,  the sketch $\mS\cdot \mA$ can be computed in $\mathcal{O}(\nnz(\mA)\log k/\eps)$ time, and with the magical graph ($s=2$), $\mS\cdot \mA$ can be computed in $\mathcal{O}(\nnz(\mA))$ time. 

\section{Theory}\label{sec:theory}

In this section, we present the theoretical results that show that the sketching matrices obtained from sparse graphs such as expanders are 
subspace embeddings under certain conditions. First, we give the definition of subspace embedding.
We use the notation $a= (1\pm \eps)b$ to mean $a\in[(1-\eps)b,(1 +\eps)b]$. 
\begin{definition}[Subspace embedding] A sketching matrix $\mS\in \R^{m\times n}$ is a $(1\pm \eps)$ $\normltwo$-subspace embedding for a subspace $\mathcal{U} \subseteq \R^n$ of dimension $k$, if for an orthonormal basis $\mU\in\R^{n\times k}$ of  $\mathcal{U}$ and for all $\vx\in \R^k$, we have
\begin{equation}
    \|\mS\mU\vx\|^2_2 = (1\pm \eps) \|\mU\vx\|_2^2 = (1\pm \eps) \|\vx\|_2^2.
\end{equation}
\end{definition}
This means that $\mS$ preserves the norm of any vectors from the $k$-dimensional subspace. It also implies the singular values of $\mS\mU$ are bounded  as $|1 - \sigma_i(\mS\mU)| \leq \eps\quad i\in[1,\ldots,k] $.

\begin{theorem}[Magical graph]
A sketching matrix $\mS\in \R^{m\times n}$ obtained from a random bipartite graph with $s =2$ is $(1\pm \eps)$ $\normltwo$-subspace embedding with probability $1-\delta$ if $m = O\left(\frac{k^2}{\delta \eps^2}  \right)$.
\end{theorem}
\begin{proof}
We briefly sketch the proof here. First, we consider the JL moment property defined in~\cite{kane2014sparser}. A sketch matrix 
$\mS\in \R^{m\times n}$  has the $(\eps,\delta,\rho)$-JL moment property, if for all $\vx\in \R^n$, with $\|\vx\|_2 =1$, we have
\[
\mathbb{E}|\|\mS\vx\|_2^2 - 1|^{\rho}\leq \eps^{\rho}\delta.
\]
We next show that our graph sketch matrix $\mS$ with $s=2$ satisfies  the$(\eps_0, \delta,2)$-JL moment property for $m = O\left(\frac{1}{\delta \eps_0^2}  \right)$.
For this, we adapt the proof of \cite[Theorem 15]{kane2014sparser}, where the $(\eps,\delta,\rho)$-JL moment property is proved for a random bipartite graph (where the edges for each left vertex and the signs $\pm1$ are assigned randomly via. a hash function) with $m$ right vertices and left degree $s$. From that proof, by setting $Y= \|\mS\vx\|_2^2 - 1$, we have
\[
\mathbb{E}|Y^{\rho}|\leq \frac{1}{s^{\rho}}\cdot 2^{\mathcal{O}(\rho)}\cdot \sum_{q=1}^{\rho/2} \rho^{\rho}\cdot \left(\frac{s^2}{qm}\right)^q.
\]
For the magical graph of~\cite{cheung2013fast}, we have $s = 2$, and if we set $\rho = 2$ and $m=\frac{16}{\delta \eps_0^2}  $, we obtain the $(\eps_0, \delta,2)$-JL moment property.
We then use the proof of \cite[Theorem 9]{woodruff2014sketching} to establish the subspace embedding property. In particular, \cite[Thm 13]{woodruff2014sketching} shows any sketch matrix $\mS$ with  $(\eps,\delta,\rho)$-JL moment property also satisfies, for many matrices  $\mA,\mB$, 
\[
Pr[\|\mA^{\top}\mB - \mA^{\top}\mS^{\top}\mS\mB\|_F>\eps \|\mA\|_F\|\mB\|_F]\leq \delta.
\]
By setting, the two matrices to be $\mU$ and $\eps = \eps_0/k$, we get the subspace embedding property.
\end{proof}
These results and the sketch size order  match that of Countsketch~\cite{clarkson2017low} that has $s=1$. However, in practice, we note that sparse graph sketch matrices with $s=2$ perform slightly better than Countsketch, see section~\ref{sec:results}.
Our next result shows that the expander graphs can achieve improved sketch sizes.

\begin{theorem}[Expander]
A sketching matrix $\mS\in \R^{m\times n}$ obtained from an expander graph with $s = O\left(\frac{\log (k/\delta\eps)}{\eps}\right)$ is $(1\pm \eps)$ $\normltwo$-subspace embedding with probability $1-\delta$ if $m = O\left(\frac{k \log(k/\delta\eps)}{ \eps^2}  \right)$.
\end{theorem}
\begin{proof}
We first use  the result from~\cite{kane2014sparser} to show that the graph sketch matrix $\mS$ is JLT. A sketch matrix is JLT$(\eps,\delta)$, if for any $\vx\in \R^n$, with $\|\vx\|_2 =1$, we have
\[
Pr[|\|\mS\vx\|_2^2 - 1|>\eps]<\delta.
\]
Results in~\cite{kane2014sparser} show that  sketch matrix from a random bipartite graph (where again the edges for each left vertex and the signs $\pm1$ are assigned randomly via. a hash function) with degree $s = \mathcal{O}(\log(1/\delta)/\eps)$ and $m = \mathcal{O}(\log(1/\delta)/\eps^2)$ is JLT. Note that, such random bipartite graphs are also expanders with appropriate parameters~\cite{hoory2006expander}.
We next use the result by   Sarlos  in~\cite[Corollary 11]{sarlos2006improved}, that showed that any JLT with $m = O\left(\frac{k \log(k/\delta\eps)}{ \eps^2}  \right)$ satisfies
\[
Pr[\forall i \in [1..k]: |1- \sigma_i(\mS\mU)|\leq \eps]\geq1-\delta.
\]
A similar result was proved using a different approach in~\cite{nelson2013osnap} for sparse sketch matrices.
 \end{proof}
 
Using the relations in~\cite{nelson2013osnap,kane2014sparser}, we can show that the above results for expanders hold even when the random signs $\pm$ assignment is just $\mathcal{O}(\log k)$-wise independent, as well as the $s$ edges assignment for each left vertex (the column permutations) can also be $\mathcal{O}(\log k)$-wise independent. This implies, we can use the code based expanders~\cite{guruswami2009unbalanced,alon1994random,cheraghchi2010derandomization}
to construct sparse sketching  matrices with reduced randomness (almost deterministic). A code matrix obtained from a coding scheme with dual distance $\gamma$ will have columns that are $\gamma$-wise independent~\cite{ubaru2017low} (the $\gamma$-wise independence property of codes). Hence, an expander with the above $s,m$ values constructed from a code with dual distance $\mathcal{O}(\log k)$ will yield a sparse sketching matrix that is subspace embedding.

\begin{figure*}[t]

\includegraphics[width=0.31\textwidth]{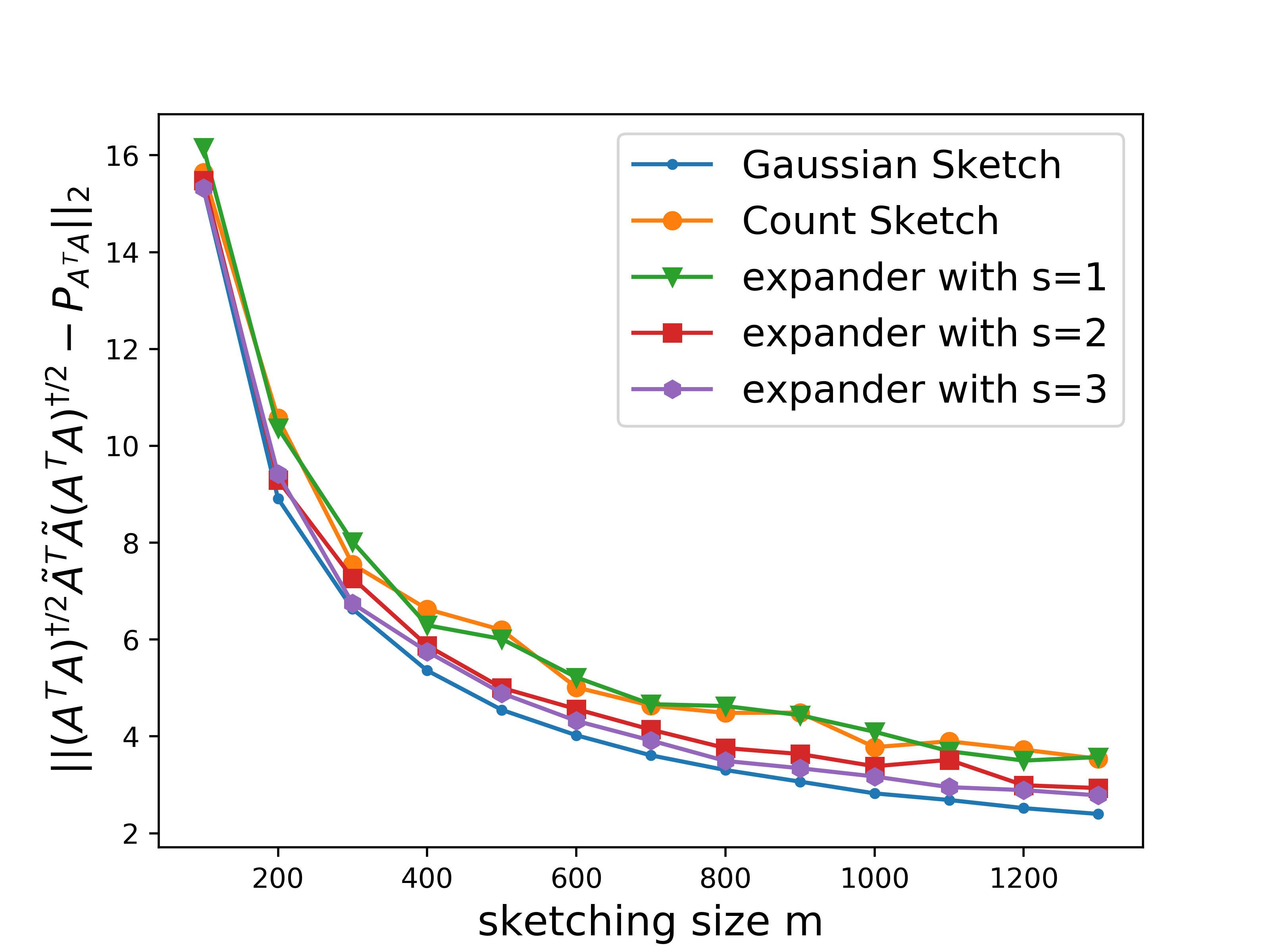}
\includegraphics[width=0.31\textwidth]{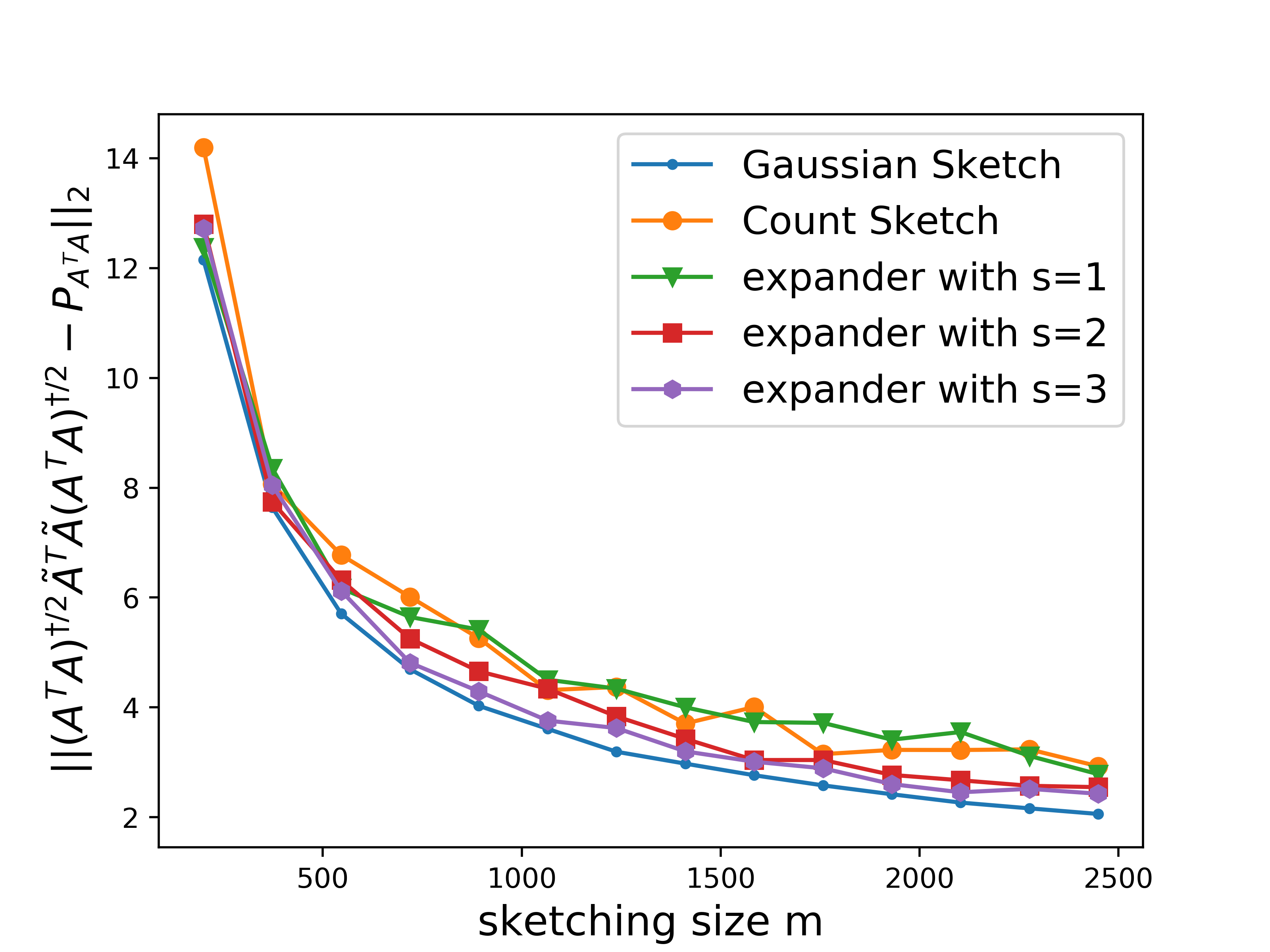}
\includegraphics[width=0.31\textwidth]{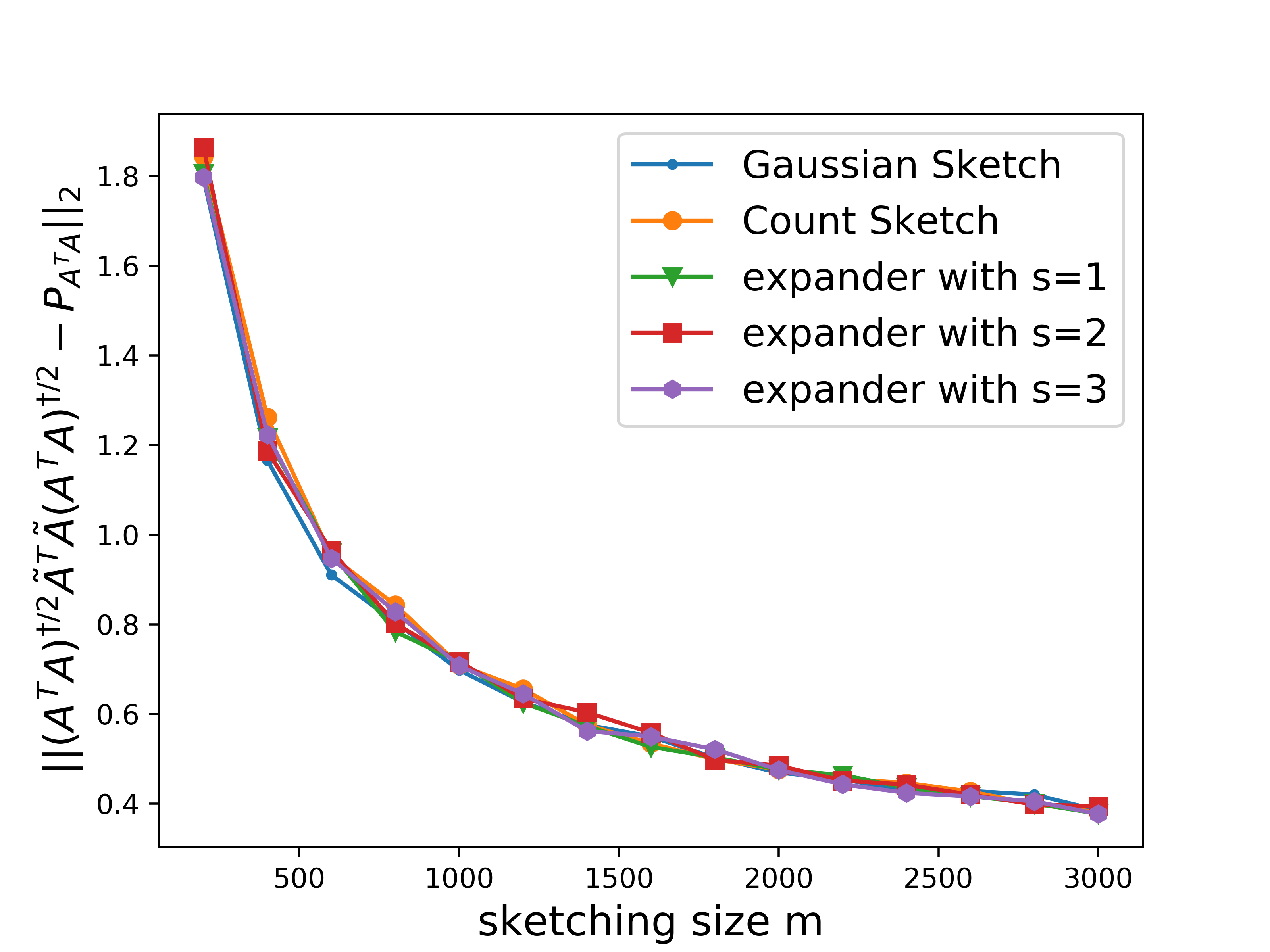}\\
\includegraphics[width=0.31\textwidth]{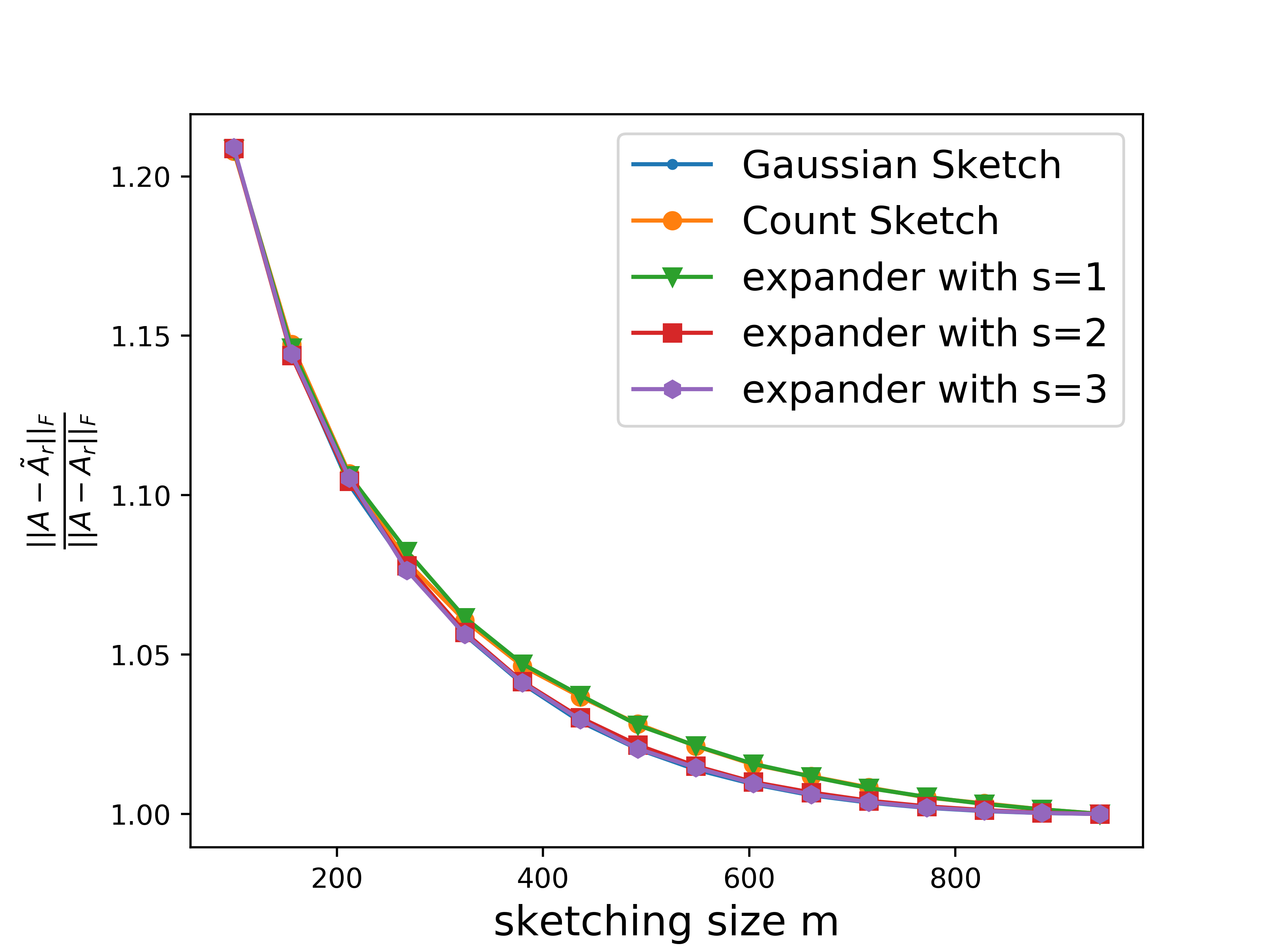}
\includegraphics[width=0.31\textwidth]{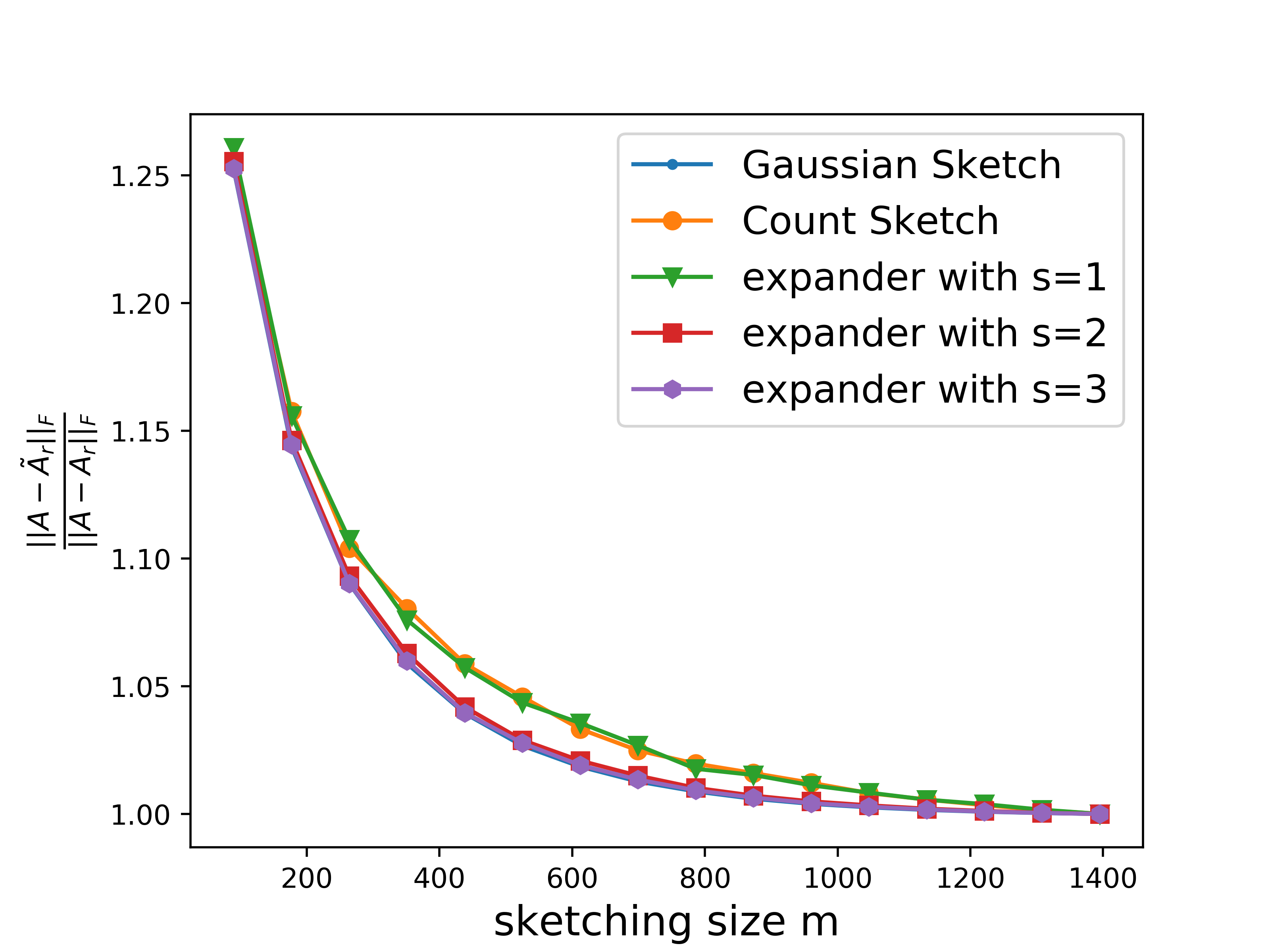}
\includegraphics[width=0.31\textwidth]{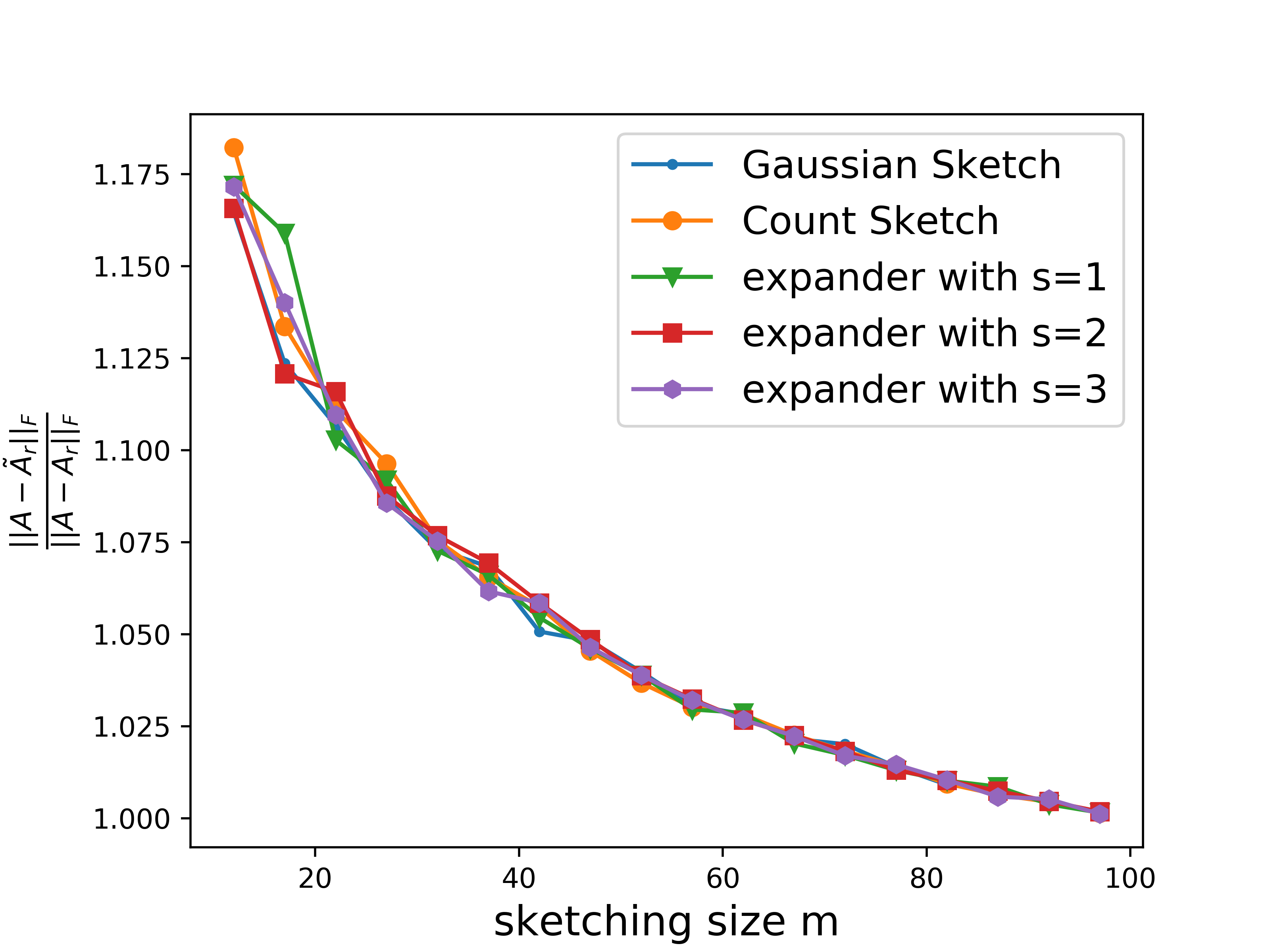}
\vskip-0.1in
\caption{The top three images plot the distortion against sketching size $m$ and the bottom 3 plots are the relative low-rank approximation error against sketching size. The first pair of plots (top and bottom left) are for the Movielens100K dataset, the second pair of plots are for the CRAN (LSI) dataset, and the last pair is for Jester dataset. Results averaged over 10 trials. }\label{fig:main}
\end{figure*}
\section{Numerical Results}\label{sec:results}
In this section, we present numerical results that illustrate the performance of the  expander (random bipartite) graph based sketching  matrices on three real datasets, and compare them with two of the most commonly used sketching matrices, viz., Countsketch~\cite{clarkson2017low} and Gaussian sketch~\cite{woodruff2014sketching}. The three datasets are: Movielens 100k dataset~\cite{harper2010}, which is a sparse matrix with size $1682 \times 943$; Jester dataset~\cite{goldberg2000}, which is a dense dataset with size $7200\times 100$ (both from the matrix completion application); and CRAN~\cite{LSIsaad}, a Latent Semantic Indexing (LSI) dataset, which is sparse with size $5204 \times 1398$. The distortion metric used in our results is defined as:

\begin{definition}[Distortion of $\tilde{\mA}$] The distortion $\eta$ of a sketched matrix $\tilde{\mA} = \mS\mA$ with respect to the original matrix $\mA$ is calculated by $\eta=\left\|\mI-\left(\mA^{{\top}} \mA\right)^{-1 / 2} \tilde{\mA}^{{\top}} \tilde{\mA}\left(\mA^{{\top}} \mA\right)^{-1 / 2}\right\|_{2}$.
\end{definition} 
For further details, we refer to page 3 of ~\cite{magdon2019fast}. Figure~\ref{fig:main} presents numerical results. The performance of the expander graph based matrices is comparable to the popular Countsketch and Gaussian sketch matrices, and the performance slightly improves as we increase the degree $s$. The performance is similar for the low-rank approximation application as well. Note that the magical graph construction  in~\cite{cheung2013fast} is the same as the expander with $s=2$ in our figure. Since SRFT and SRHT also yield similar results on these datasets, we elected to exclude them for visual clarity sake.  
The sketch matrices will have different size $n$ as per the original data matrix.

For the low-rank approximation, we apply the following steps after sketching: WLOG, we sketch the row space of the input matrix $\mA$, and we get $\mY = \mS\mA$, and we compute QR factorization to get $\mQ,\mR={\text qr}(\mY^{\top})$. Next, we construct $\mB=\mA\mQ$, compute  rank-$k$ SVD of $\mB$, $[\mU_k,\mathbf{\Sigma}_k,\mV_k]={\text SVD}(\mB, k)$ and set $\mV_k=\mQ\mV_k$. And then the low-rank approximation is calculated from the formula $\|\mA-\mA\mV_k\mV_k^{\top}\|_F$, and we denote ${\tilde \mA}_k=\mA\mV_k\mV_k^{\top}$ for the $y$ axis in the bottom 3 plots,and $\mA_k$ is the true rank-$k$ approximation.

\subsection*{Conclusions and future work}
In this study, we proposed sparse graph based sketching matrices for fast numerical linear algebra computations. We presented theoretical results that show that under certain conditions,  these matrices satisfy the subspace embedding property. Numerical results illustrated the performance of different sketching matrices in applications. We have made our implementation publicly available as a \emph{Python Matrix-Sketching toolkit}: \url{https://github.com/jurohd/Matrix-Sketching-Toolkits}. As part of our future work, we plan to explore randomized sketching techniques for nonlinear operators, with applications to deep learning.

\subsection*{Acknowledgements} This work was supported by the Rensselaer-IBM AI Research Collaboration (\url{http://airc.rpi.edu}), part of the IBM AI Horizons Network (\url{http://ibm.biz/AIHorizons}).

%{\small
%\bibliographystyle{abbrv}
%\bibliography{refs}
%}
{\small
\input{magicgraph.bbl}}
\end{document}

%% file: math_commands.tex
\usepackage{amsmath,amsfonts,bm}

\def\eqref#1{equation~\ref{#1}}
\def\1{\bm{1}}

\def\eps{{\epsilon}}

\def\vb{{\bm{b}}}

\def\vx{{\bm{x}}}

\def\mA{{\bm{A}}}
\def\mB{{\bm{B}}}

\def\mI{{\bm{I}}}

\def\mQ{{\bm{Q}}}
\def\mR{{\bm{R}}}
\def\mS{{\bm{S}}}

\def\mU{{\bm{U}}}
\def\mV{{\bm{V}}}

\def\mY{{\bm{Y}}}

\DeclareMathAlphabet{\mathsfit}{\encodingdefault}{\sfdefault}{m}{sl}
\SetMathAlphabet{\mathsfit}{bold}{\encodingdefault}{\sfdefault}{bx}{n}

\def\gG{{\mathcal{G}}}

\newcommand{\R}{\mathbb{R}}

\newcommand{\nnz}{\mathrm{nnz}}
\newcommand{\normltwo}{\ell_2}

\DeclareMathOperator*{\argmin}{arg\,min}